\documentclass[letterpaper, 10pt, conference]{ieeeconf}
\IEEEoverridecommandlockouts
\usepackage{amssymb}
\usepackage{amsmath}
\usepackage{stix}
\usepackage{tikz}
\usepackage{algorithm}
\usepackage{algpseudocode}

\usepackage{amsthm}
\usepackage{xcolor}
\usepackage{cite}
\usepackage{hyperref}
\usepackage{booktabs}
\usepackage{multirow}

\usetikzlibrary{positioning, shapes, arrows, calc, decorations.markings}
\tikzstyle{block} = [draw, rectangle, line width=0.7mm]
\tikzstyle{sum} = [draw, circle, line width=0.5mm, minimum width=3mm, inner sep=0mm]

\newcommand{\tran}{{\mkern-1.5mu\mathsf{T}}}

\DeclareMathOperator{\diag}{diag}
\newcommand{\bmat}[1]{\begin{bmatrix} #1 \end{bmatrix}}
\theoremstyle{definition}
\newtheorem{definition}{Definition}
\newtheorem{theorem}{Theorem}
\newtheorem{prop}{Proposition}
\newtheorem{lemma}{Lemma}[theorem]
\newtheorem{corollary}{Corollary}[theorem]
\newtheorem{remark}{Remark}
\newtheorem{conjecture}{Conjecture}

\def\H2{\mathcal{H}_2}
\def\L1{\mathcal{L}_1}
\def\R{\mathbb{R}}

\def\D{\mathcal{D}}

\def\bDelta{\boldsymbol{\Delta}}
\def\bdelta{\boldsymbol{\delta}}

\def\nubar{\overline{\nu}}
\def\deltabar{\overline{\delta}}
\def\Hinf{\mathcal H_\infty}

\def\dyndeltaset{\mathscr{D}}
\def\statdeltaset{\mathscr{C}}
\def\commutants{\mathcal{D}}

\title{\LARGE \bf
$\nu$-Analysis: A New Notion of Robustness for Large Systems with Structured Uncertainties
}

\author{Olle Kjellqvist and John C. Doyle
	\thanks{O. Kjellqvist is with the Department of Automatic Control, Lund University. J. C. Doyle is with Computing and Mathematical Sciences, California Institute of Technology. {\tt\small olle.kjellqvist@control.lth.se, doyle@caltech.edu}. O.K has received funding from the European Research Council (ERC) under the European Union’s Horizon 2020 research and innovation programme under grant agreement No 834142
(ScalableControl).
}
}

\begin{document}

\maketitle
\thispagestyle{empty} 
\pagestyle{empty}

\begin{abstract}

We present a new, scalable alternative to the structured singular value, which we call $\nu$, provide a convex upper bound, study their properties and compare them to $\ell_1$ robust control.
The analysis relies on a novel result on the relationship between robust control of dynamical systems and non-negative constant matrices.

\end{abstract}
\section{Introduction}
We consider a system to be robust if it is unlikely to fail.
The usual setting to analyze the robustness of a system is to study how it interacts with uncertainty.
Standard approaches impose structure on the uncertainty and certify robustness against its size.
However, the way we currently measure the size of uncertainty is unsuitable for large-scale networks.
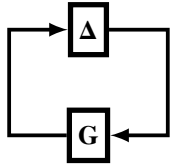
\begin{figure}
    \centering
    \begin{tikzpicture}[node distance = 4em, line width=0.5mm, inner ysep=6pt, >=latex]
	\node [block, align = center] (sys) {$\mathbf G$};
	\node[block, above of = sys, align = center] (delta) {$\bDelta$};
	\draw[->] (sys) -- ++(-3em, 0) |- (delta);
	\draw[<-] (sys) -- ++(3em, 0) |- (delta);
\end{tikzpicture}
    \caption{Interconnection of the stable system $\mathbf G$ and uncertainty $\bDelta$ considered for robust stability.}
    \label{fig:robust_stability}
\end{figure}

To see this, consider the standard robust control set-up in Fig.~\ref{fig:robust_stability}.  $\mathbf G$ is a stable causal linear system with $n$ inputs and outputs. $\bDelta$ is unknown but belongs to the set $\dyndeltaset$ consisting of diagonal linear time-varying (LTV) systems that are causal, stable, and have $n$ inputs and outputs. We want to determine which of the following two systems are most likely to fail.

\begin{equation*}
\begin{aligned}
    \mathbf P_1:& \begin{cases}
            x_1(t+1) = \delta_{1}x_1(t)\\
            \qquad \vdots \\
            x_n(t+1) = \delta_{n}x_n(t)
            \end{cases},\ 
    \mathbf P_2:& \begin{cases}
            x_1(t+1) = \delta_{1}x_2(t)\\
            \qquad \vdots \\
            x_n(t+1) = \delta_{n}x_1(t).
            \end{cases}
\end{aligned}
\end{equation*}

$\mathbf P_1$ is a set of decoupled first-order systems with uncertain time constants, and $\mathbf P_2$ is a delayed ring with uncertain weights. Robustness measures based on structured singular values~\cite{Zhou1998Robust, Dullerud2010Robust} or $\ell_1$ robust control methods~\cite{Dahleh1993Controller} agree that both systems are robust against diagonal uncertainties whose \emph{largest}\footnote{In $\Hinf$-- and $\ell_1$--norm respectively} diagonal element is bounded by one.
It is tempting to conclude that $\mathbf P_1$ and $\mathbf P_2$ are equally likely to fail.

A more careful study reveals that destabilizing $P_1$ is easy; a constant gain $|\delta_k| > 1$ for any $k$ will render the closed-loop unstable. However, all of the uncertainties must simultaneously be large $(\|\delta_1\|\|\delta_2\| \cdots \|\delta_n\| \geq 1$ to destabilize $P_2$. 
In plain words, destabilizing $\mathbf P_2$ requires large globally coordinated perturbations directly affecting every node.

This article proposes a new robustness measure $\nu$\footnote{The robustness measure $\nu$ is unrelated to Vinnicombe's $\nu$-gap metric. We apologize for the confusion caused by overloading $\nu$ and highlight the need for further research into new Greek letters.} that captures sparsity in the uncertainty. 
It is large for dense uncertainties and small when sparse uncertainties can destabilize the system.
For example, $\nu(\mathbf P_1) = 1$ and $\nu(\mathbf P_2) = 1/n$.
We focus on diagonal linear time-varying and nonlinear uncertainty in discrete time.

This work is primarily motivated by recent progress to distributed and localized controller design for large-scale networks~\cite{Anderson2019}, modeling and analysis of the feedback in neuroanatomy~\cite{acc1, acc2, acc3} and the need for better control methods for emerging large-scale systems such as smart-grids and intelligent transportation systems.

\subsection{Outline}
Section~\ref{sec:Prelims} introduces notation and gives some background on robust stability for static and dynamic matrices. We introduce and analyze the new robustness measure in Section~\ref{sec:nu} and provide a convex upper bound.
Section~\ref{sec:nubar_analysis} describes the properties of the upper bound and in Section~\ref{sec:nubar_computation} we show how to compute it and characterize the optimal solution.
Concluding remarks and directions for future research are contained in Section~\ref{sec:conclusions}.

\section{Preliminaries and notation}
\label{sec:Prelims}
Latin letters denote real-valued vectors matrices like $x \in \R^n$ and $A \in \R^{n\times m}$.
For a matrix $A \in \R^{n \times m}$, $A_{ij}$ means the element on the $i$th row and $j$th column, and we refer to the $i$th element of a vector $x \in \R^n$ by $x_i$.
The $p$-norm of a vector $x \in \R^n$ is defined by
\[
    |x|_p := \begin{cases}
                \left(\sum_{i = 1}^n |x_i|^p\right)^{1/p} & \text{if } p \in [0, \infty),\\
                \max_i |x_i| & \text{if } p = \infty.
            \end{cases}
\]

For a matrix, $A \in \R^{n \times m}$, the induced norm from $q$ to $p$ is defined by
\[
    |A|_{q, p} : = \max_{x}\frac{|Ax|_p}{|x|_q}.
\]
For an infinite sequence $\mathbf x = \{x(0), x(1), \ldots\}, x(k) \in \R^n$, $m_x = (|x_1|_\infty, \cdots, |x_n|_\infty)$ is called the magnitude vector of $\mathbf x$ and $\ell_\infty^n$ denotes the set of all such sequences such that
\[
    \|\mathbf x\|_\infty := |m_x|_\infty < \infty.
\]
We define the truncation operator $P_T$ on $\ell_\infty^n$ by
\[
    P_T(x) = (x(0), \ldots, x(T), 0, \ldots).
\]
By $\ell_{\infty, e}^n$ we mean the extended $\ell_\infty^n$-space: $\{x \in \ell_\infty^n: (P_Tx) \in \ell^n_\infty: T \geq 0 \}$.
An operator $H : \ell_\infty^n \to \ell_\infty^n$ is causal if $P_kH = P_kHP_k$ and time-invariant if it commutes with the delay operator $z^{-1}$.
We call $\mathbf H$ $\ell_\infty$ stable if
\[
    \|\mathbf H\|_1 := \|\mathbf H\|_{\infty, \infty} = \sup_t\sup_{0 \neq x \in \ell_{\infty,e}}\frac{\|P_t \mathbf H x\|_\infty}{\|P_tx\|_\infty} < \infty,
\]
where $\|\mathbf H\|_{\infty,\infty}$ is called the induced norm on $\ell^n_\infty$.
A linear time-varying operator $\mathbf H$ is fully characterized by it's impulse response (convolution kernel) $H(t, \tau)$ and it operates on signals $\mathbf x \in \ell_{\infty, e}$ by convolution, $(\mathbf Hx)(t) = \sum_{\tau = 0}^t H(t, \tau)x(\tau)$.
Expressed in the elements of its convolution kernel, the induced norm becomes $\|\mathbf H\|_{1} = \max_i\sum_{j = 1}^n\sup_t\sum_{\tau = 0}^t |H_{ij}(t, \tau)|$.
It will be convenient to express the norm in terms of the magnitude matrix of $\mathbf H$
\begin{equation}
    \label{eq:magnitude_matrix}
    M_H := \begin{bmatrix}
    \|H_{11}\|_1 & \cdots & \|H_{1n}\|_1 \\
    \vdots & \ddots & \vdots \\
    \|H_{n1}\|_1 & \cdots & \|H_{nn}\|_1 
    \end{bmatrix},
\end{equation}
and $\|\mathbf H\|_1 = \max_i \sum_j M_{ij}$.
\subsection{Matrix induced norms and stability of static systems}
Before diving into induced norms for dynamical systems, we explore norms on constant $n$-dimensional vectors and square matrices.
For a constant matrix $M\in \R^{n\times n}$, robust stability with respect to bounded \emph{unstructured} uncertainty means that $\det(I - \Delta M)$ is invertible for all $|\Delta| \leq \gamma$ in some norm.
Let $\frac{1}{p} + \frac{1}{q} = 1$, then $\det(I - \Delta M)$ is invertible for all $|\Delta|_{p, q} \leq 1$ if and only if $|M|_{q, p} < 1$.
See Table~\ref{tab:tropptable} for a table of the most common compatible $p$-norms. 
\begin{table}
    \vspace{1em}
    \centering
    \caption{Summary of matrix induced norms, adapted from~\cite{Tropp2004Topics}. The norm on the domain (D) is determined by the column, and the codomain (CD) by the row.}
    \label{tab:tropptable}
    \begin{tabular}{cccc}
        \toprule CD$\backslash$D& $|\cdot|_1$ & $|\cdot|_2$ & $|\cdot|_{\infty}$  \\
        \midrule
         $|\cdot|_1$ & $\max_{j}\sum_{i = 1}^n |A_{ij}|$ & NP-HARD & NP-HARD\\
         $|\cdot|_2$ & $\sqrt{\max_{j}\sum_{i = 1}^n |A_{ij}|^2}$ & $\overline\sigma(A)$ & NP-HARD\\
         $|\cdot|_\infty$ & $\max_{ij}|A_{ij}|$ & $\sqrt{\max_{i}\sum_{j = 1}^n |A_{ij}|^2}$ & $\max_{i}\sum_{j = 1}^n |A_{ij}|$ \\
         \bottomrule
    \end{tabular}
\end{table}

\subsection{Robust stability with diagonal uncertainty}
Let $\dyndeltaset$ be the set of $\ell_\infty$-stable causal linear time-varying operators whose off-diagonal elements are zero, and $\D\subset \R^{n \times n}$ be the set of diagonal matrices with positive diagonal entries.
Further, define
\begin{equation}
    \label{eq:mu}
    \mu_\dyndeltaset(\mathbf G) = \frac{1}{\inf\{\|\bDelta\|_{\infty, \infty} : \bDelta \in \dyndeltaset,\ (I - \mathbf G\bDelta)^{-1} \text{ unstable} \}}.
\end{equation}
The following Theorem characterizes robust stability of Fig.~\ref{fig:robust_stability} as conditions on $M_G$.

\begin{theorem}[Theorem 2 in \cite{Dahleh1993Controller}]
\label{thm:DhaKha}
For $\bDelta \in \dyndeltaset$ with $\|\bDelta\|_{\infty, \infty} \leq 1$, the following are logically equivalent :
\begin{enumerate}
    \item The system in Fig. \ref{fig:robust_stability} is robustly stable.
    \item $\rho(M_G) < 1$, where $\rho(\cdot)$ denotes the spectral radius.
    \item $x \leq M_G x$ and $x \geq 0$ imply that $x = 0$.
    \item $\inf_{D \in \commutants}|D M_G D^{-1}|_{\infty, \infty} < 1$
    \item $\mu_\dyndeltaset(\mathbf G) < 1$.
\end{enumerate}

\end{theorem}

\begin{remark}
As $\rho(A) = \rho(A^\top)$ we have that the convex upper bound (which is exact for linear time-varying uncertainty) can be computed either as the max row sum or max column sum. That is,  $\inf_{D \in \commutants} |DM_G D^{-1}|_{\infty, \infty} = \inf_{D \in \commutants} |DM_G D^{-1}|_{1,1}$.
\end{remark}
\section{\texorpdfstring{$\nu$}{}: The new \texorpdfstring{$\mu$}{}}\label{sec:nu}
Inspired by the role of LASSO~\cite{Tibshirani96Regression} in favoring sparse solutions to regression problems, we propose using the sum of $\ell_1$ norms, $\sum_{i = 1}^n \|\bdelta_i\|_1$.
The new robustness metric $\nu$ is defined as follows.
\begin{definition}[$\nu$]
\label{def:nu}
    Let $\dyndeltaset$ be the set of $\ell_\infty$-stable causal linear time-varying operators with $n$ inputs and outputs, whose off-diagonal elements are zero.
    Given a causal linear system $\mathbf G$ with $n$ inputs and outputs
    \[
    \nu_\dyndeltaset(\mathbf G) := \frac{1}{\inf\{\sum_{i=1}^n\|\bdelta_i\|_{1} : \bDelta \in \dyndeltaset,\ (I - \mathbf G\bDelta)^{-1} \text{ unstable}\}}.
    \]
\end{definition}

To study the properties of the new robustness measure and its relationship to $\mu$, we require insight into the relationship between that of destabilizing the dynamical system $\mathbf G$ and its magnitude matrix $M_G$. From Theorem~\ref{thm:DhaKha} we know that if there exists a $\bDelta$ that destabilizes Fig.~\ref{fig:robust_stability}, then there exists a constant matrix $\Delta$ with the same $\ell_1$ norm so that $I - \Delta M_G$ is singular, and vice versa.
Surprisingly, it turns out that the bounds on each diagonal entry of $\bDelta$ are equal to that of $\Delta$.
\begin{theorem}
\label{thm:equivalence}
Let $\dyndeltaset$ be the set of $\ell_\infty$-stable causal linear time-varying operators with $n$ inputs and outputs, whose off-diagonal elements are zero.
Further, let $\statdeltaset \subset \R^{n \times n}$ be the set of non-negative diagonal matrices.
Given upper bounds $\deltabar_{ii}$ for $i=1, \ldots, n$ and a stable, causal $n \times n$-dimensional system  $\mathbf G$, the following are logically equivalent:
\begin{enumerate}
    \item There exists a $\bDelta \in \dyndeltaset$, where each diagonal element is bounded from above, $\|\bdelta_{ii}\|_{\infty,\infty} \leq \deltabar_{ii}$, such that the system in Fig.~\ref{fig:robust_stability} is unstable.
    \item There exists a $\Delta \in \statdeltaset$, where each diagonal element is bounded from above, $\delta_{ii} \leq \deltabar_{ii}$, such that $I-\Delta M_G$ is singular.
\end{enumerate}
\end{theorem}
\begin{proof}
    We start by showing that the first claim implies the second.
    Let $R_\Delta = \diag(\bar \delta_{11}, \ldots, \overline \delta_{nn})$, then $\bDelta = \widehat{\bDelta} R_\Delta$ for some $\widehat{\bDelta} \in \dyndeltaset, \|\widehat{\bDelta}\|_1 \leq 1$. As Fig.~\ref{fig:robust_stability} with $\bDelta \mathbf G = \widehat{\bDelta}(R_\Delta \mathbf G)$ is unstable, we conclude the existence of a diagonal non-negative matrix $\widehat \Delta$ with $|\widehat \Delta|_{\infty, \infty} \leq 1$ so that $I - \widehat \Delta R_\Delta M_G$ is singular.
    Taking $\Delta = \widehat{\Delta} R_\Delta$ completes the first part of the proof.
        
    The proof of the converse result is identical but starts with $M_G$.
\end{proof}

Theorem~\ref{thm:equivalence} implies that we can replace the $\ell_1$ norm in \eqref{eq:mu} with any norm on the magnitude matrix of $\bDelta$ and get $\mu_\dyndeltaset(\mathbf G) = \mu_\statdeltaset(M_G)$ for free. 

Although we do not yet know how to compute $\nu$, from Fig.~\ref{fig:robust_stability} and Theorem~\ref{thm:equivalence} we know that $\nu$ must be absolutely homogeneous and invariant to similarity transforms with matrices that commute with $\dyndeltaset$.
Furthermore, we can translate the equivalence relationship between $|\cdot|_1$ and $|\cdot|_\infty$ into a corresponding relationship between $\nu$ and $\mu$.
We summarize the above discussion with the following proposition:
\begin{prop}
\label{prop:nu}
    With $\mathbf G$, $\dyndeltaset$, $\statdeltaset$ as in Theorem.~\ref{thm:equivalence}, let $\commutants \subset \R^{n\times n}$ be the set of non-negative diagonal matrices, then the following statements are true:
    \begin{enumerate}
        \item $\nu_\dyndeltaset(\mathbf G) = \nu_\statdeltaset(M_G)$
        \item $\nu_\dyndeltaset(a\mathbf G) = |a| \nu_\dyndeltaset(\mathbf G)$ for $a \in \R$.
        \item $\nu_\dyndeltaset(D\mathbf GD^{-1}) = \nu_\dyndeltaset(\mathbf G)$ for $D \in \commutants$.
        \item $\mu_\dyndeltaset(\mathbf G) / n \leq \nu_\dyndeltaset(\mathbf G) \leq \mu_\dyndeltaset(\mathbf G)$.
    \end{enumerate}
\end{prop}

The following theorem tightens the lower bound in 4) by zeroing out different diagonal elements.
This result agrees with intuition because we can study how a system interacts with sparse uncertainty by testing the different sparsity patterns separately.

\begin{theorem}
    \label{thm:lower_bound}
    Given $\mathbf G$ and $\dyndeltaset$ as in Definition~\ref{def:nu}.
    Let $I = (i_1, i_2, \ldots, i_m)$ with $m \leq N$ and $i_k \neq i_l$ for $k \neq l$ be an index tuple, and consider the sub-matrix of $M_G$:
    \begin{equation}
    \label{eq:subM}
        M_I = \bmat{
        M_{i_1i_1} & \cdots & M_{i_1i_m} \\
        \vdots & \ddots & \vdots \\
        M_{i_mi_1} & \cdots & M_{i_mi_m}
        }.
    \end{equation}
    Then $\nu_\dyndeltaset(\mathbf G) \geq \frac{\rho(M_I)}{m}$.
\end{theorem}

\begin{proof}
    Assume without loss of generality that $i_k = k$ for $k = 1, \ldots, m$. This assumption can always be enforced by renaming the signals. Restrict $\bDelta$ by setting $\bdelta_{kk} = 0$ for $k > m$. By Proposition~\ref{prop:nu} $\nu_\dyndeltaset(\mathbf G) = \nu_\dyndeltaset(M_G)$, so it is sufficient to give the proof in the constant matrix case. Let $\Delta_1 = \diag(\delta_{11}, \ldots, \delta_{mm})$ be the submatrix of $\Delta$ that is nonzero and partition $M_G$ into
    \[
    M_G = \bmat{
    M_{11} & M_{12} \\
    M_{21} & M_{22}
    },
    \]
    where $M_{11} \in \R^{m \times m}$. 
    Thus $I - \Delta M_G$ is invertible if and only if $(I - \Delta_1 M_{11})$ is invertible, which is equivalent to $|\Delta_1|_{\infty,\infty} \leq 1/\rho(M_I)$.
    From the fourth property of Proposition~\ref{prop:nu} we conclude that $\nu_\dyndeltaset(\mathbf G) \geq \rho(M_I)/(m)$.
\end{proof}

\subsection{An upper bound of \texorpdfstring{$\nu$}{}}

If the norm on the magnitude matrix of $\Delta$ is one in the upper triangle of Table~\ref{tab:tropptable}, then we can use the corresponding dual norm in the lower triangle to construct an upper bound.

Although the induced norm from $\infty$ to $1$, in general, is NP-hard to compute, it coincides with the absolute sum for diagonal matrices.
To see this, consider
\begin{equation}
        |\Delta|_{\infty, 1} = \sup_{|x|_\infty = 1}\sum_{i = 1}^n|\delta_{ii}x_i| = \sum_{i = 1}^n|\delta_{ii}|.
\end{equation}

This implies that if $|M_G|_{1, \infty} < 1 / |\Delta|_{\infty, 1}$ then $I - \Delta M_G$ is non-singular.
As $\nu_\dyndeltaset$ is invariant under similarity transformations with $D \in \commutants$, we suggest the following upper bound:
\begin{equation}
\label{eq:nubar}
    \nubar_\dyndeltaset(\mathbf G) := \inf_{D \in \commutants} |DM_GD^{-1}|_{1, \infty}
\end{equation}

The $1$ to $\infty$ norm is the maximum absolute element of a matrix, see Table~\ref{tab:tropptable}, and can be computed for large-scale connected systems by local evaluation and communication with the closest neighbors.

\begin{remark}
    The $(1, \infty)$ and $(\infty, 1)$-norms on magnitude matrices $M_G$ and $M_\Delta$ correspond to norms on $\mathbf G$ and $\mathbf \Delta$ induced by the $1$ and $\infty$-norms on the magnitude vectors. That is, let $y = \mathbf G x$, then $|m_y|_\infty \leq |M_G|_{1, \infty}|m_x|_1$ and if $z = \mathbf \Delta v$ then $|m_z|_1 \leq |M_\Delta|_{\infty, 1}||m_v|_\infty$. One could start with this observation, define $\nu$ with respect to the induced norm on $\mathbf \Delta$, and get the same results as in this paper. We prefer the simplicity of Theorem~\ref{thm:equivalence}.
\end{remark}

We end this section by noting that for positive systems, the $\Hinf$-norm is achieved by a stationary input~\cite{Rantzer2015Scalable, Colombino2016Convex}, so robustness analysis can be done entirely on positive matrices in that case too. We suspect one can derive similar results for positive systems as those in this article.
\begin{conjecture}
    For positive systems, there exists a convex upper bound for a robustness measure against a causal, diagonal, linear time-varying uncertainty $\bDelta$ bounded in the following norm $\|\bDelta\| = \|\bdelta_1\|_\infty + \|\bdelta_2\|_\infty + \cdots + \|\bdelta_n\|_\infty$.
\end{conjecture}

\section{Properties of \texorpdfstring{$\nubar$}{}.}\label{sec:nubar_analysis}

The lower bound in Theorem~\ref{thm:lower_bound} shows that if the maximum absolute value is achieved on the diagonal of $M_G$, then the upper bound coincides with the lower bound and is exact.
These types of systems are called \emph{diagonally maximal} and merit a formal definition.

 \begin{definition}[Diagonally Maximal]
    A Matrix $A \in \mathbb{R}^{n \times n}$ is \textit{diagonally maximal} if the maximum absolute element of $A$ appears on the diagonal. A dynamical system $\mathbf G$ is diagonally maximal if its magnitude matrix $M_G$ is diagonally maximal.
 \end{definition}

The following important corollary follows from applying Theorem~\ref{thm:lower_bound} to each diagonal element.
\begin{corollary}
    \label{cor:diagonally_maximal}
    If the matrix $DM_G D^{-1}$ is \emph{diagonally maximal} for some $D \in \D$, then $\nubar_\dyndeltaset(\mathbf G) = \nu_\dyndeltaset(\mathbf G)$.
\end{corollary}

Going back to the systems $\mathbf P_1$ and $\mathbf P_2$ in the introduction, we see that $\mathbf P_1$ is diagonal and hence diagonally maximal and $\nu_\dyndeltaset(\mathbf P_1) = \nubar_\dyndeltaset(\mathbf P_1) = 1$.
However, for $\mathbf P_2$ the upper bound is conservative. Indeed, $1/n = \nu_\dyndeltaset(\mathbf P_2) \leq \nubar_\dyndeltaset(\mathbf P_2) = 1 $.
The following theorem describes the gap between $\nu$ and $\nubar$.
\begin{theorem}
With $\nu$, $\mathbf G$ and $\dyndeltaset$ as in Definition~\ref{def:nu} and $\nubar$ as in \eqref{eq:nubar}, it is true that
    \[
        1 \leq \frac{\nubar(\mathbf G, \dyndeltaset)}{\nu(\mathbf G, \dyndeltaset)} \leq n.
    \]
    Furthermore, the lower bound is achieved by systems $\mathbf G$ that are \emph{diagonally maximal} under some similarity transform $D$ that commutes with $\dyndeltaset$.
    Pure rings achieve the upper bound.
\end{theorem}
\begin{proof}
    By construction $\nubar(\mathbf G, \dyndeltaset) \geq \nu(\mathbf G, \dyndeltaset)$, and by Corollary~\ref{cor:diagonally_maximal} the upper bound is exact for systems that are diagonally maximal under some similarity transform that commutes with $\dyndeltaset$.

    By $|M_G|_{1, \infty} \leq |M_G|_{\infty, \infty}$ and Proposition~\ref{prop:nu} we have that $\nubar_\dyndeltaset(\mathbf G) \leq \mu_\dyndeltaset(\mathbf G) \leq n\nu_\dyndeltaset(\mathbf G)$.
    It remains to show that the upper bound is achieved for pure ring systems.
    After scaling, balancing, and relabeling the signals, a pure ring system is of the form
        \begin{align*}
            x_1(t+1) = \delta_{11}x_2(t),\quad \cdots,\quad x_n(t+1) & = \delta_{nn}x_1(t).
        \end{align*}
By Proposition~\ref{prop:nu} $\nu_\dyndeltaset(G) = \nu_\statdeltaset(M_G)$, so we will study the null space of $I - M_G\Delta$.
$I - M_G\Delta$ has a nontrivial null space if for some non-zero $w\in \R^n$,
\[
    (I - M_G\Delta)w = 0 \iff \bmat{
    w_1 - \delta_{22}w_2 \\
    w_2 - \delta_{33}w_3 \\
    \vdots \\
    w_n - \delta_{11}w_1
    }
    = 0.
\]
If $w_1 = 0$, then by substitution we must have $w = 0$.
So assume without loss of generality that $w_1 = 1$.
Then we have that $I - M_G\Delta$ has a nontrivial null space if and only if 
\begin{equation}
    \label{eq:prod_delta}
    \delta_{11}\cdots \delta_{nn} = 1.
\end{equation}

We proceed to lower bound $\sum_{i = 1}^n \delta_{ii}$ by minimizing it subject to \eqref{eq:prod_delta}.
Substitute $\delta_{nn} = 1 / \prod_{i=1}^{n-1} \delta_{ii}$ into the sum to transform the constrained optimization problem into a convex optimization problem over $\delta_{ii} > 0$ with the solution $\min_{\delta_{ii}}\sum_{i=1}^n\delta_{ii} = n$.
Substitute the lower bound on a destabilizing $\Delta$ into Definition~\ref{def:nu} to get
\[
\frac{\nubar(\mathbf G, \dyndeltaset)}{\nu(\mathbf G, \dyndeltaset)} \geq n,
\]
as $\nubar_\dyndeltaset(\mathbf G) = 1$.
Since the upper bound is equal lower bound, we conclude that the bound is achieved.
\end{proof}

By the discussion in this section it is clear that even though $\nubar$ bounds $\nu$, the gap can be pretty significant.
It stands to reason that $\nubar$ is exact for some class of disturbances.
\begin{conjecture}
    $\nubar$ is exact for some class of norm-bounded disturbances.
\end{conjecture}

We conclude this section by studying $2\times 2$ matrices.

\subsection{A closed-form formula for \texorpdfstring{$2\times 2$}{} matrices}

As $\nu$ is invariant under similarity transforms with $D \in \D$ and scaling, we consider matrices $M \in \R^{2 \times 2}$ of the form
\[
    M = \bmat{x & 1\\ 1 & y}.
\]
If $x> 1$ or $y > 1$ we know that $\nu_\statdeltaset(M) = \max\{x, y\}$ so we will consider the case $0 < x, y < 1$.
We begin by parameterizing all destabilizing $\Delta$ in $\delta_{22}$.
Setting the determinant to zero we get
\begin{align*}
    \frac{1}{\det(M)}\left( y + \frac{-1}{x - \det(M)\delta_{22}}\right) & = \delta_{11}.
\end{align*}
Thus $\nu_\statdeltaset(M) = \delta_{11}(\delta_{22}) + \delta_{22} $ is convex on the domain $[0, 1]$ and the minimum is achieved either on the boundary or at a stationary point.
For $0 < x, y < 1$ we get
\begin{equation}
    \delta_{11} = \frac{y - 1}{\det(M)}, \quad \delta_{22} = \frac{x - 1}{\det(M)}, \nu(M, \dyndeltaset) = \frac{\det(M)}{x + y - 2}.
\end{equation}

In Fig.~\ref{fig:ratios} we compare the new robustness metric $\nu$, the upper bound $\nubar$ and $\mu$ for $2\times 2$-matrices. We see that $\nubar$ is exact for and only for matrices that are diagonally maximal under some $D \in \commutants$ and conclude that even for diagonally maximal systems, $\nu$ and $\mu$ can be very different.
As the closed-loop maps generated by system-level synthesis often seem to be diagonally maximal, we conclude that for a large class of relevant systems, computing both $\nubar$ and $\mu$ gives additional information into the nature of destabilizing disturbances even for this class of systems.
Based on this observation we state the following conjecture.
\begin{conjecture}
    $\nubar_\statdeltaset(M) = \nu_\statdeltaset(M)$ only if $DMD^{-1}$ is diagonally maximal for some $D \in \statdeltaset$.
\end{conjecture}

\begin{figure}
    \centering
    \includegraphics[width=.3\textwidth]{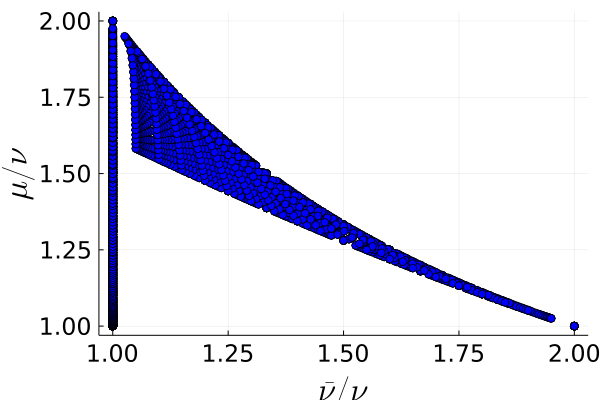}
    \caption{Comparing the $\ell_1$-robustness metric $\mu$, the new metric $\nu$ and the upper bound $\nubar$ for matrices of the form $M = \bmat{x & w \\ w & y}$ for $x, w, y \in [0, 1]$. The matrices along the $\nubar/\nu = 1$ line are the diagonally maximal matrices. In the bottom left corner we have the identity matrix, in the top left corner we have the matrix $\bmat{1 & 1 \\ 1 & 1}$ and in the bottom right corner we have $\bmat{0 & 1 \\ 1 & 0}$.}
    \label{fig:ratios}
\end{figure}

\section{Computing \texorpdfstring{$\nubar$}{}}\label{sec:nubar_computation}

\subsection{The convex approach}
This section explains how to formulate $\nubar$ as a linear program.
Let $M \in \R^{n \times n}$ be a positive matrix.
We want to compute 
\begin{equation}
    \label{eq:nu_comp}
    \inf_{D \in \D}  \max_{ij}\left \{M_{ij}\frac{d_i}{d_j}\right\},
\end{equation}
As the logarithm is strictly increasing, \eqref{eq:nu_comp} is equivalent to
\[
\min_{D \in \D} \max_{ij}\{\log(M_{ij}) + \log(d_i) - \log(d_j)\},
\]
where we use the convention that $\log(0) = -\infty$.
Let $\beta_i = \log(d_i)$, then \eqref{eq:nu_comp} is equivalent to the following linear program that can be solved efficiently using simplex or interior-point methods~\cite{Todd2002Linear}:
\begin{equation}
\label{eq:nu_linprog}
\begin{aligned}
    \underset{\beta, \gamma}{\text{minimize}} \quad & \gamma \\
    \text{subject to:} \quad & \log(M_{ij}) + \beta_i - \beta_j \leq \gamma \\
    & \beta_i \in \R\ \text{for}\ i=1, \ldots, n, \quad \gamma \in \R.
    \end{aligned}
\end{equation}

\subsection{Characterizing the solutions of the upper bound}
We will relax the positivity assumption of $d_1, \ldots, d_n$ \eqref{eq:nu_comp} to allow $d_i$s to be zero.
Consider the function
\begin{equation}\label{eq:phi}
        \phi_d(M, i, j) = \begin{cases}
        M_{ij}\frac{d_i}{d_j} & \text{if } M_{ij} > 0 \\
        0 & \text{if } M_{ij} = 0.
        \end{cases}
\end{equation}
Then \eqref{eq:nu_comp} is equivalent to
\begin{equation}\label{eq:nu_comp_nonneg}
    \inf_{d_1, \ldots, d_n \geq 0}\max_{ij}\phi_d(M, i, j).
\end{equation}

The following theorem shows that if for some $D\in \commutants$, the maximizing indices of $DMD^{-1}$ only consists of loops, then $D$ minimizes \eqref{eq:nu_comp_nonneg}.
\begin{theorem}[Sufficient condition for optimality]\label{thm:suff}
    Given a non-negative, non-zero matrix $M\in \R^{n\times n}$ and non-negative constants $d_1, \ldots, d_n$. With $\phi$ as in \eqref{eq:phi}, let $\mathcal I$ be the set of maximizing indices of \eqref{eq:nu_comp}, i.e.
    \[
        \mathcal I = \left \{ (k, l) : \phi_d(M, k, l) = \max_{ij}\phi_d(M, i, j) \right \}.
    \]
    If for all $(k, l) \in \mathcal I$ it holds that
    \begin{equation}
        \phi_d(M, k, l) = \max_{j}\phi_d(M, l, j).
    \end{equation}
    Then $d_1, \ldots, d_n$ is an optimal solution to \eqref{eq:nu_comp}.
\end{theorem}
\begin{proof}
    First, we show that $\mathcal I$ must contain at least one loop.
    Let $(j_0, j_1) \in \mathcal I$, and let $j_{k + 1}$ be the smallest integer such that 
    \[
    \phi_d(M, j_k, j_{k+1}) = \max_{j}\phi_d(M, j_k, j)
    \]
    By induction $(j_k, j_{k+1}) \in \mathcal I$.
    Furthermore, as $n$ is finite, and the selection rule for $j_{k+1}$ is unique given $j_k$, there is a $K \geq 0$ and a $T \geq 1$ so that $j_{k + T} = j_{k}$ for all $k \geq K$.
    We denote the limit set containing such points by
    $I_* = \{j_k : k \geq K\}$.
    
    Assume towards a contradiction that there are $d_1', \ldots, d_n'$ so that 
    \[
        \max_{ij}\phi_{d'}(M, i, j) < \max_{ij}\phi_d(M, i, j)
    \]
    Let $(j_0, j_1) \in \mathcal I_*$. 
    Assume without loss of generality that $d'_{j_{1}} > d_{j_{1}}$, otherwise multiply every $d'_i$ by a positive constant so that the assumption holds true.
    Let $j_2 = \arg\max_j \phi_d(M, j_1, j)$.
    By assumption, it must hold that 
    \[
        \frac{d'_{j_1}}{d'_{j_2}} < \frac{d_{j_1}}{d_{j_2}} \iff d'_{j_2} > d_{j_2}\frac{d'_{j_1}}{d_{j_1}}.
    \]
    Continuing, we have that
    \[
    d'_{j_{k+T}} > d_{j_{k+T}}\frac{d'_{j_{k + T-1}}} {d_{j_{k + T - 1}}} > d_{j_{k+T}}\frac{d'_{j_{k}}}{d_{j_{k}}}.
    \]
    However, since $j_{k + T} = j_k$ we have that $d'_k > d'_k$ which is a contradiction.
\end{proof}

By the above theorem, we know that if the maximum is achieved on a loop, then the solution is optimal.
It turns out that an optimal solution must contain a loop.
This is because if the maximum is achieved on a chain, we can perturb the scales at the end of the chain to make that value smaller.
Since this shortens the chain.
Repeating this process reduces all the elements in the maximal chain.
We formalize this statement in the following Lemma:

\begin{lemma}
    \label{lemma:loop}
    Let $d^\star_1, \ldots, d^\star_n$ be an optimal solution to \eqref{eq:nu_comp_nonneg}, and let $\mathcal I$ be the set of maximizing indices as in Theorem~\ref{thm:suff}. Then $\mathcal I$ contains at least one loop.
\end{lemma}

\begin{proof}
    If the optimal value is zero, all diagonal elements must be zero, and $(i, i) \in \mathcal I$ implies that $\mathcal I$ contains a loop.
    Assume towards a contradiction that $\mathcal I$ does not contain a loop and that the optimal value is greater than zero.
    Let $(j_0, j_1) \in \mathcal I$, and let $j_{k + 1}$ be the smallest integer such that 
    \[
    \phi_d(M, j_k, j_{k+1}) = \max_{j}\phi_d(M, j_k, j)
    \]
    By assumption there is a $k$ such that 
    \begin{equation}\label{eq:induct_decrease}
    \phi_d(M, j_k, j_{k+1})< \max_{i}\phi_d(M, i, j_k)
    \end{equation}
    This means that there is a $d'_k > 0$ that decreases the right hand side of \eqref{eq:induct_decrease} so that the inequality still holds for $j_k$, but also holds for $j_{k-1}$.
    By induction, this must hold for $1, \ldots, k$.
    Repeating for any other chain in $\mathcal I$, we conclude that
    \[
        \max_{ij}\phi_d(M, i, j) > \max_{ij}\phi_{d'}(M, i, j),
    \]
    contradicting optimality.
\end{proof}

Theorem~\ref{thm:suff} and Lemma~\ref{lemma:loop} indicate a relationship between solving \eqref{eq:nu_comp_nonneg} and balancing the matrix $M$ with respect to the maximal absolute element.
The following theorem strengthens that connection and shows that we can always find a solution to \eqref{eq:nu_comp_nonneg} by balancing $M$.
\begin{theorem}
    For any non-negative matrix $M \in \R^{n \times n}$, there exists a non-negative solution $d_1, \ldots, d_n$ to \eqref{eq:nu_comp_nonneg} such that
    \begin{equation}\label{eq:existance}
        \max_{r\neq k}\phi_d(M, r, k) = \max_{c\neq k}\phi_d(M, k, c),\quad \forall\ k = 1, \ldots, n.
    \end{equation}
\end{theorem}

\begin{proof}
    We begin by proving the existence of a solution.
    Assume there is a sequence $i_1, \ldots, i_m$ such that $M_{i_ki_{k+1}}, M_{i_mi_1} \neq 0$ for $k = 1, \ldots, m$. 
    Then \eqref{eq:nu_comp} is bounded below by $\min\{M_{ij}, M_{ji}\}$ and \eqref{eq:nu_comp} is equivalent to a linear program with a bounded solution and the minimum is achieved by some $d_1, \ldots, d_n$.
    If the assumption is false, we can take $d = 0$ and the optimal value is zero.
    If $M$ is a diagonal matrix, then the claim holds trivially. Assume $M$ is not diagonal and let $\hat M$ be the matrix where $\hat M_{ij} = M_{ij}$ for $i \neq j$ and $\hat M_{ii} = 0$.
    Then $d_1, \ldots, d_n$ are optimal for $M$ if and only if they are optimal for $\hat M$.
    Note that \eqref{eq:existance} holds for a maximizing loop of $\hat M$.
        Let $d_1, \ldots, d_n$ be an optimal solution to $\eqref{eq:nu_comp}$ for $\hat M$.
    By Lemma~\ref{lemma:loop}, the set of maximizing indices $\mathcal I$ contains at least one loop.
    Remove the rows and columns pertaining the loop from $\hat M$ to get the smaller matrix $\hat M_1$.
    By recursion on $\hat M_k$ we end up with a new set $d'_1, \ldots, d'_n$ so that \eqref{eq:existance} is true.
\end{proof}

\subsection{An algorithm for balancing the magnitude matrix}
We now present a simple heuristic algorithm for computing \eqref{eq:nu_comp} that results from enforcing \eqref{eq:existance} coordinate-wise in Algorithm~\ref{alg:nu_comp}.
The algorithm lends itself to local computation and we show some empirical convergence properties in figures \ref{fig:tol} and \ref{fig:size}.
\begin{algorithm}
\caption{Heuristic algorithm for solving \eqref{eq:nu_comp}}\label{alg:nu_comp}
\begin{algorithmic}
    \Require Non-negative $M \in \R^{n \times n}$, $\theta \in (0, 1)$, $T$.
    \State $d_k[1] \gets 1$ \textbf{for} each $k = 1, \ldots, n$
    \For{$t = 1, \ldots, T$}
        \For{$k = 1, \ldots, n$}
            \State $
            d_k[t + 1] \gets (1 - \theta) d_k[t] + \theta\frac{\sqrt{\max_{r\neq k} M_{rk}d_r[t]}}{\sqrt{\max_{c\neq k} M_{kc}/d_c[t]}}
            $
        \EndFor
    \EndFor
\end{algorithmic}
\end{algorithm}
We remark that naively taking $\theta = 1$ may cause the algorithm to fail to converge.
Consider the matrix,
\[
    M = \bmat{0 & 1 \\ x^2 & 0}.
\]
Then $d_1(2) = x $ and $d_2(2) = 1/x$, leading to $D(2) M D^{-1}(2) = M^\tran$ and the iteration will continue to oscillate back and forth.
This is because we are updating each coordinate simultaneously, which is desirable for localized computation.
Introducing the interpolation $\theta \in (0, 1)$ seems to solve this issue. Based on the numerical results we conjecture that our algorithm has is guaranteed to converge.
\begin{figure}
    \centering
    \vspace{1em}
    \includegraphics[width=.3\textwidth]{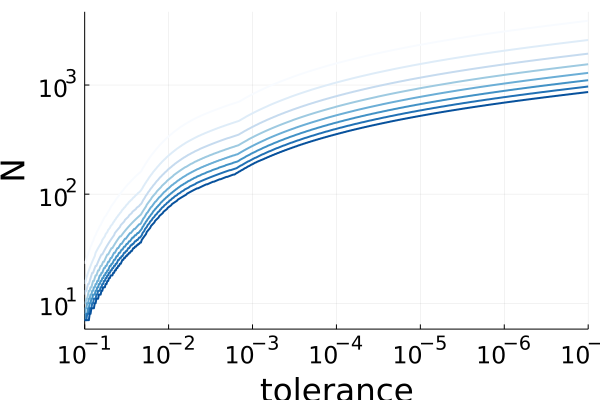}
    \caption{The largest number of iterations $N$ required to reach a relative tolerance level for $100$ randomly generated non-negative matrices $M \in \R^{128 \times 128}$, with respect to tolerance. $\theta$ ranges from $0.2$ (light blue) to $0.9$ (dark blue).}
    \label{fig:tol}
\end{figure}
\begin{figure}
    \centering
    \vspace{1em}
    \includegraphics[width=.3\textwidth]{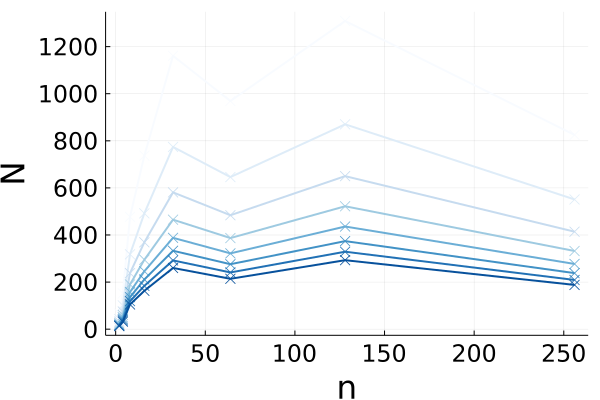}
    \caption{The largest number of iterations $N$ required to reach a relative tolerance level of $10^{-3}$ for $100$ randomly generated non-negative $n\times n$-dimensional matrices with respect to dimension $n$. $\theta$ ranges from $0.2$ (light blue) to $0.9$ (dark blue).}
    \label{fig:size}
\end{figure}

\begin{conjecture}
    Algorithm~\ref{alg:nu_comp} always converges.
    Moreover the number of iterations required to reach a given tolerance is of $O(log(n))$ for a fixed $\epsilon$, and $O(\sqrt{\epsilon^{-1}})$ for fixed $n$.
\end{conjecture}
\section{Conclusions}\label{sec:conclusions}
This work introduced and analyzed a new robustness measure $\nu$ that reasonably handles sparsity.
We provided a convex upper bound $\nubar$, characterized its sub-optimality, and gave simple ways to compute it in a distributed way.
The companion paper, \cite{CompanionPaper_SLS} shows how to compute robust controllers for large-scale systems using $\mu$ and $\nu$. Throughout this article, we gave four conjectures representing important research topics. We conclude with a final conjecture on the computation of $\nu$.

\begin{conjecture}
    There exists a polynomial-time algorithm to compute $\nu$ within arbitrary precision.
\end{conjecture}
\bibliography{main}
\bibliographystyle{IEEEtran}

\end{document}